\documentclass[11pt]{amsart}

\usepackage{amsfonts}
\usepackage{amscd}
\usepackage{amssymb}

\allowdisplaybreaks

\usepackage{hyperref}

\date{\today}

\newtheorem{thm}{Theorem}[section]
\newtheorem{cor}[thm]{Corollary}

\newtheorem{prop}[thm]{Proposition}
\theoremstyle{definition}

\theoremstyle{remark}
\newtheorem{rem}[thm]{Remark}

\numberwithin{equation}{section}

\newcommand{\R}{\mathbb R}
\newcommand{\m}{\bf{m}}
\newcommand{\Z}{\mathbb Z}
\newcommand{\T}{\mathbb T}

\newcommand{\C}{{\mathbb C}}

\usepackage{color}

\title[Holomorphic Sobolev spaces]
{Segal-Bargmann transforms and \\Holomorphic Sobolev spaces of fractional order}

\author[ S. Thangavelu]{ Sundaram Thangavelu}

\address[S. Thangavelu]{Department of Mathematics\\
Indian Institute of Science\\
560 012 Bangalore, India}

\email{veluma@iisc.ac.in}

\begin{document}

\maketitle

\begin{abstract} In this note we investigate the image of Sobolev spaces  of fractional order on a compact Lie group $ K $ under the Segal-Bargmann transform. We show that the image can be characterised  in terms of certain weighted Bergman spaces of holomorphic functions on the complexification extending a theorem of Hall and Lewkeeratiyutkul. We also treat the heat kernel transform associated to the Hermite operator.

 \end{abstract}

\section{Introduction} 
Let $ K $ be a connected, compact Lie group with Lie algebra $ \mathfrak{k}.$ Fix an $ Ad$-$K $ invariant inner product $ \langle , \rangle $ on $ \mathfrak{k}$ and choose an orthonormal basis $ X_1, X_2,....,X_n $ which are viewed as left invariant vector fields on $ K.$  Then $ \Delta = \sum_{j=1}^n X_j^2 $ turns out to be the Laplace-Beltrami operator  for the bi-invariant metric determined by $ \langle , \rangle $  on $K $. Let $ G $ stand for the universal complexification of $ K $ which is a complex Lie group whose Lie algebra is given by $ \mathfrak{k}+i\mathfrak{k}.$ For any irreducible unitary representation $ \pi $ of $ K $ on a Hilbert space $ H_\pi $ we let $ d_\pi $ stand for the dimension of $ H_\pi.$ For any $ f \in L^2(K) $ we let
$$ \pi(f) = \int_K f(k) \pi(k) dk $$
stand for the Fourier transform of $ f.$  The Plancherel theorem reads as
$$ \int_K |f(k)|^2 dk = \sum_{\pi \in \widehat{K}} d_\pi \| \pi(f)\|^2 $$
where $ \widehat{K} $ is the unitary dual of $ K $ and $ \| \pi(f)\| $ stands for the Hilbert-Schmidt norm of $ \|\pi(f)\|.$

Any $ \pi \in \widehat{K} $ can be analytically continued to $ G $  whose entries are then holomorphic functions on $ G.$  Let $ dg $ stand for the Haar measure on $ G$. By $ \mathcal{H}L^2(G,\nu dg) $ we mean $ L^2(G,\nu dg)\cap \mathcal{O}(G) $  for any density function $ \nu(g) $ on $ G.$ In his celebrated paper \cite{H} Hall has proved the following Paley-Wiener type theorem for compact Lie groups.

\begin{thm} [Hall] Let $ \nu $ be a $K$ bi-invariant function on $ G $ which is locally bounded away from zero and assume that for any $ \pi \in \widehat{K} $ the integrals
$ \sigma(\pi) = d_\pi^{-1} \int_G \|\pi(g^{-1})||^2  \nu(g) dg $ are finite. Then $ f \in L^2(K) $ has a holomorphic extension $ F $ to $ G $ which belongs to $\mathcal{H}L^2(G,\nu dg) $ if and only if $$ \sum_{\pi \in \widehat{K}} d_\pi \|\pi(f)\|^2 \sigma(\pi) < \infty.$$ Moreover, if that happens then we have the equality of norms:
$$ \int_G |F(g)|^2 \nu(g) dg = \sum_{\pi \in \widehat{K}} d_\pi \|\pi(f)\|^2 \sigma(\pi).$$
\end{thm}

The above theorem contains the result on Segal-Bargmann transform on $ K.$ Let $ q_t $ stand for the heat kernel associated to $ \frac{1}{2}\Delta$ which is explicitly given by
$$ q_t(k)  =  \sum_{\widehat{K}} d_\pi  e^{-\frac{1}{2}t\lambda_\pi^2} \chi_\pi(k) $$ 
where $ \chi_\pi $ is the character of $ \pi $ and $ \lambda_\pi $ is determined by the condition $ \Delta \chi_\pi(k) = -\lambda_\pi^2 \chi_\pi(k).$ It is known that $ q_t(k) >0 $ and satisfies $ q_t(k) = q_t(k^{-1}) $ and $ q_t(kk') = q_t(k'k).$ Because of these two properties $ f \ast q_t = q_t \ast f $ for any $ f \in L^2(K) $ and $ u(k,t) = f \ast q_t(k,t) $ solves the heat equation for $ \frac{1}{2}\Delta $ with initial condition $ f.$ It is known that $ q_t $ extends to $ G $ as a holomorphic function and consequently  for any $ f \in L^2(K) $ the solution $ f \ast q_t $ also has a holomorphic extension to $ G .$ The map $ C_t $ which takes $ f $ into the holomorphic function $ F(g) = f \ast q_t(g) $ is known as the Segal-Bargmann transform.

The Lie algebra $ \mathfrak{g} $ of $ G $ viewed as a real Lie algebra has dimension $ 2n $ and the inner product $ \langle , \rangle $ on $ \mathfrak{k} $ can be extended to $ \mathfrak{g} $ by setting
$$ \langle X_1+JY_1, X_2+JY_2 \rangle =\langle X_1, X_2 \rangle + \langle Y_1,Y_2 \rangle, \,\, X_i,Y_i \in \mathfrak{k} $$
where $ J $ stands for multiplication by $ i.$ 
Then $ \Delta_G $ turns out to be a left invariant differential operator on $ G $ which is  $ K $ bi-invariant. Let $ \mu_t(g) $ stand for the heat kernel associated to $ \frac14 \Delta_G$
so that $ \mu_t \ast f $ solves the heat equation. We  define $ \nu_t(g)  = \int_K \mu_t(kg) dk$  which is $K$ bi-invariant. Viewed as a subgroup of $ G , K $ is maximal compact and so $ G/K $ becomes a noncompact Riemannian symmetric space equipped with the left invariant metric defined by the above inner product.  We can identify  $ \Delta_G $ with the  Laplace-Beltrami operator on $ G/K$ and $ \nu_t $ is the bi-invariant heat kernel.

With these notations we can calculate that
\begin{equation} d_\pi^{-1} \int_G \|\pi(g^{-1})||^2  \nu_t(g) dg = e^{t\lambda_\pi^2} \end{equation}
and Theorem 1.1 leads to the following result for the Segal-Bargmann transform: for any $ f \in L^2(K) $ 
$$  \int_G | f \ast q_t(g)|^2 \nu_t(g) dg = \int_K |f(k)|^2 dk.$$
Thus the map $ C_t $ which takes $ f $ into $ f \ast q_t(g) $ is an isometry from $ L^2(K) $ into $ \mathcal{H}L^2(G,\nu_t dg) $ and it is also known that this map is onto. We can restate this as follows: A holomorphic function $ F $ on $ G $  belongs to $ L^2(G,\nu_t dg) $ if and only if $ F = f \ast q_t $ for some $ f \in L^2(K).$ This characterises the image of $ L^2(K) $ under the Segal-Bargmann transform.

Using spectral theorem we can define the fractional powers $ (-\Delta)^{s/2} $ for any $ s \in \R.$ The domain of this operator is the (homogeneous) Sobolev space $ H^s(K) $ consisting of distributions $ f $ for which 
$$  \sum_{\pi \in \widehat{K}} d_\pi\, \lambda_\pi^{s}\, \| \pi(f)\|^2 < \infty.$$
For $ s\geq 0 $ we note that $ H^s(K) \subset L^2(K) $ and hence we can restrict $ C_t $ to $ H^s(K) $ and ask for a characterisation of its image under the Segal-Bargmann transform.
In \cite{HL} Hall and Lewkeeratiyutkul have considered this problem for $ s= 2m $ an even integer. They have introduced holomorphic Sobolev spaces $ \mathcal{H}^{2m}(G,\nu_t) $
and shown that $ C_t $ takes $ H^s(K ) $ onto $\mathcal{H}^{2m}(G,\nu_t).$ These Sobolev spaces are defined as follows. By considering $ \Delta $ as a left invariant differential operator on $ G $ and so we can talk about $ \Delta^{m}F $ for any holomorphic function on $ G.$  Note that these are defined and holomorphic on $ G $ but it is not necessarily true that they are in $ L^2(G, \nu_t) $ even if we start with $ F \in \mathcal{H}L^2(G,\nu_t).$ By defining $\mathcal{H}^{2m}(G,\nu_t) $ to be the subspace of $ \mathcal{H}L^2(G,\nu_t) $ consisting of those $ F $ for which $  \Delta^{m}F \in \mathcal{H}L^2(G,\nu_t).$

\begin{thm}[Hall-Lewkeeratiyutkul] For any non-negative integer $ m $ the Segal-Bargmann transform $ C_t $ takes $ H^{2m}(K) $ onto $ \mathcal{H}^{2m}(G,\nu_t) .$ Moreover, it is possible to find a positive weight function $ w_{2m}(g) $ such that $  \mathcal{H}^{2m}(G,\nu_t) = \mathcal{H}L^2(G, w_{2m} \nu_t).$
\end{thm}

The same problem for $ s < 0 $ has been taken up in \cite{T3} where it has been shown that the image of $ H^s(K) $ is a weighted Bergman space on $ G.$ For non-integral $ s >0 $ the problem remains open. The aim of this note is precisely to address this problem.  In order to define holomorphic Sobolev spaces of fractional order we make use of the complexified  Laplacian $ \Delta_\C $ acting on $\mathcal{H}L^2(G,\nu_t).$ For  any  $ \gamma > 0 $ we define a weight function
$$  w_{t,\gamma}(g)   = \frac{1}{\Gamma(2\gamma)} \int_0^t (t-r)^{2\gamma-1} \, \nu_r(g) \,dr $$ 
Given $ s > 0 $ fix  any integer $ m $ such that $ 0 <s < 2m.$  We say that a function $ F $ from $ \mathcal{H}L^2(G,\nu_t)$ belongs to $
\mathcal{H}_0^s(G,\nu_t) $ if and only if $ (-\Delta_\C)^m F \in \mathcal{H}L^2(G, w_{t,m-s/2}).$ We equip this space with the norm defined by
$$ \|F||_{(s)}^2  = \int_{G} |F(g)|^2 \nu_t(g) + \int_G |(-\Delta_\C)^mF(g)|^2 w_{t,m-s/2}(g) dg .$$ We remark that as $ s \rightarrow 2m $ the weight function $ w_{t,m-s/2} $ converges to $ \nu_t $ and hence the new definitions coincides with the original definition given in \cite{HL}. With this definition we have

\begin{thm} For any  $ s \geq 0$  the Segal-Bargmann transform $ C_t : H^{s}(K) \rightarrow  \mathcal{H}_0^{s}(G,\nu_t) $ is an isomorphism. 
\end{thm}

In the above setting $ C_t $ is no longer an isometry. However, on $ H^s(K) $ we can introduce an equivalent norm so that we get an isometric isomorphism. This is easy to achieve. For any $ \gamma > 0 $ we set
$$  \mu_{t,\gamma}(\pi) = \frac{1}{\Gamma(\gamma)} \int_0^t r^{\gamma-1} e^{-r \lambda_\pi^2} dr$$
and note that as $ \lambda_\pi $ goes to infinity, $ \mu_{t,\gamma}(\pi) $ behaves like $ \lambda_\pi^{-2\gamma}.$ For any distribution $ f $ on $ K $ we let
$$  \|f\|_{t,\gamma}^2 = \sum_{\pi \in \widehat{K}} d_\pi \|\pi(f)\|^2 \mu_{t, 2\gamma}(\pi)$$
and define a new norm by   $ \|f\|_{(t,s)}^2  = \|f\|_2^2 + \|(-\Delta)^mf\|_{t,m-s/2}^2.$
Then it is easily seen that $ f \in H^s(K)$ if only if $ \|f\|_{(t,s)} <\infty. $ We denote by $ H_0^s(K) $ the set $ H^s(K) $ equipped with this new norm. Then we have

\begin{thm} For any  $ s \geq 0$  the Segal-Bargmann transform $ C_t : H_0^{s}(K) \rightarrow  \mathcal{H}_0^{s}(G,\nu_t) $ is an isometric isomorphism. 
\end{thm}

We  do not know if these Sobolev spaces $ \mathcal{H}_0^{s}(G,\nu_t) $ are weighted Bergman spaces.
However we can restate the above result in the following form. Given $ s > 0 $ let us choose $ m $ such that $ m-1 < s/2 \leq  m $ for the sake of definiteness and define
$$ C_t^s f(g)  = \frac{2^m}{\Gamma(m-s/2)} \int_0^t r^{m-s/2-1} \partial_r^m(f \ast  q_{r+t})(g) dr$$
which we may call the generalised (or integrated) Segal-Bargmann transform. 
Then  we can prove that $C_t^s : H^{s}(K) \rightarrow  \mathcal{H}L^2(G,\nu_t) $ is  an  isomorphism, see Remark 2.5.

So far we have considered Segal-Bargmann transforms  on compact Lie groups but such transforms, also known as heat kernel transforms have been studied in non-compact situations also, see the works \cite{KOS},\cite{KTX},\cite{T2} and \cite{MS}. These transforms are defined in terms of heat semigroups associated to Laplacians, sublaplacians and certain elliptic partial differential operators. Therefore, it is natural to study fractional order Sobolev spaces associated to such operators. For  the sake of avoiding repetition we only study the Segal-Bargmann  transform associated to the Hermite semigroup $ T_t = e^{-tH} $ where $ H = -\Delta+|x|^2 $ is the Hermite operator. Then it is known that for any $ f \in L^2(\R^n) $ the function $ F = T_tf $ has an  entire extension to $ \C^n.$ The image of $ L^2(\R^n) $ under this Segal-Bargmann or heat kernel transform has been characterised. Let 
$$ U_t(z) =  2^n (\sinh(4t))^{-n/2} e^{-(\coth 2t)|y|^2+(\tanh 2t) |x|^2}.$$
then the following theorem has been proved in \cite{B}, see also \cite{T2}. 

\begin{thm}  An entire function $ F $ on $ \C^n $ is of the form $ T_tf(z) $ for some $ f \in L^2(\R^n) $ if and only if $ F $ belongs to the weighted Bergman space $ \mathcal{H}L^2(\C^n,U_t).$ Moreover, we have the equality of norms
$$  \int_{\C^n} |F(z)|^2 U_t(z) dz = \int_{\R^n} |f(x)|^2 dx.$$
\end{thm}

As in the case of  compact Lie groups, we can also study the generalised heat kernel transform associated to the Hermite semigroup. We let
$$ T_t^s f(z)  = \frac{1}{\Gamma(m-s/2)} \int_0^t r^{m-s/2-1} \partial_r^m(f \ast  T_{r+t})(z) dr$$
where we have chosen $ m $ such that $ (m-1)\leq s/2 <m.$  Let $ W_H^{s,2}(\R^n) $ stand for the Hermite-Sobolev space of order $ s $ defined as the closure of the domain of the fractional power $ H^{s/2}.$ We have the following 

\begin{thm} For any $ s > 0 $, $ T_t^s :W_H^{s,2}(\R^n) \rightarrow \mathcal{H}L^2(\C^n,U_t)$  is an isomorphism. By equipping $ W_H^{s,2}(\R^n) $ with an equivalent norm, we can make $ T_t^s $ an isometry.
\end{thm}

We also have an analogue of Theorem 1.4 for which we need to define holomorphic analogues of the Hermite-Sobolev spaces $ W_H^{s,2}.$ We  set
$$  U_{t,\gamma}(z)   = \frac{1}{\Gamma(2\gamma)} \int_0^t (t-r)^{2\gamma-1} \, U_r(z)  \,dr $$ 
and define $ \mathcal{H}_H^{s,2}(\C^n, U_t) $ as the subspace of $ \mathcal{H}L^2(\C^n,U_t) $ consisting of  those $ F $ for which
$$  \int_{\C^n} | H_\C^m F(z)|^2 U_{t,m-s/2}(z) dz < \infty.$$
In the above $ H_\C = \sum_{j=1}^n (-\frac{\partial^2}{\partial z_j^2} +z_j^2) $ is the complexification of the Hermite operator. Here is the analogue of Theorem 1.4.

\begin{thm} For any $ s \geq 0,$ the heat kernel transform $ T_t :  W_H^{s,2}(\R^n) \rightarrow \mathcal{H}_H^{s,2}(\C^n, U_t) $ is an isomorphism.

\end{thm}

We remark that the case $ s = 2m $ of the above theorem has been already proved in \cite{RT}. The proofs of the above theorems depend on the so called Gutzmer's formula for the Hermite expansions established in \cite{T2}.

\section{Holomorphic Sobolev spaces on compact Lie groups}

We begin with some general considerations. Let $ \pi $ be any unitary  representation of the compact Lie group $ K $ on a Hilbert space $ \mathcal{H}.$  Given  $ \varphi \in L^2(K), f \in \mathcal{H}  $ and a left invariant vector field $ X $ consider
$$  \int_K  X\varphi(k) \pi(k)^\ast f dk = - \int_K \varphi(k) X \pi(k)^\ast f dk.$$
Since $$ X\pi(k)^\ast f = \frac{d}{dt}\Big{|}_{t=0} \pi(e^{-tX}k^{-1})f = - d\pi(X)\pi(k)^\ast f $$
it follows that 
$$ \int_K X\varphi(k) \pi(k)^\ast f dk = d\pi(X) \Big(\int_K \varphi(k)  \pi(k)^\ast f dk\Big).$$ 
Let $ \Delta  = \sum_{j=1}^n X_j^2 $ be the Laplacian on $ K.$ then we have
$$ \int_K \Delta \varphi(k) \pi(k)^\ast f dk = \pi(\Delta) \Big(\int_K \varphi(k)  \pi(k)^\ast f dk\Big)$$ 
where $ \pi(\Delta) = \sum_{j=1}^n d\pi(X_j)^2.$ We are interested in defining $ (-\pi(\Delta))^{s/2} $ for any $ s \geq 0.$ 

Apart from using spectral theory, there are other ways of defining fractional powers of Laplacians, see \cite{G}. One such useful definition is provided by solving the extension problem for the Laplacian. By assuming $ 0 < s < 1, $ let us consider the initial value  problem 
$$ \big( \Delta+\partial_\rho^2 +\frac{1-s}{\rho} \partial_\rho \big) u(k,\rho) = 0, \,\, u(k,0) = f(k)$$
where $ f \in L^2(K).$ An explicit solution of the above problem is given in terms of the heat semigroup $ e^{t\Delta}.$ Indeed, the function
$$ u(k,\rho) = \rho^s  \frac{1}{\Gamma(s)} \int_0^\infty e^{-\frac{\rho^2}{4t} }e^{t\Delta}f(k) t^{-s/2-1} dt $$
solves the extension problem, see \cite{ST}. If $ u $ is the solution given by the integral above, then it has been shown that $ \rho^{1-s} \partial_\rho u(k,\rho) $ converges to a constant multiple of $ (-\Delta)^{s/2}f $ as $ \rho $ tends to zero, see \cite{CS}. As $ e^{t\Delta}f = q_{2t}\ast f $ (remember: $ q_t $ is the kernel of $ e^{\frac{t}{2}\Delta} $) we can represent the solution as $ u(k,\rho) = \rho^s \varphi_{s,\rho}\ast f(k) $ where
$$ \varphi_{s,\rho}(k) = \frac{1}{\Gamma(s)} \int_0^\infty e^{-\frac{\rho^2}{4t} } q_{2t}(k) t^{-s/2-1} dt .$$
By direct calculation we can show that  $ \varphi_{s,\rho} \in L^1(K) $  and satisfies the equation 
\begin{equation} \big( \Delta+\partial_\rho^2 +\frac{1-s}{\rho} \partial_\rho \big) \big(\rho^s \varphi_{s,\rho}\big)(k) = 0 .\end{equation}
Moreover, using the spectral expansion of the heat kernel, it is not hard to show that
\begin{equation} \lim_{\rho \rightarrow 0} \rho^{1-s} \partial_\rho \big( \chi_\pi \ast \rho^s \varphi_{s,\rho})(k) = c_s \lambda_\pi^s  \chi_\pi(k) \end{equation}
for any $ \pi \in \widehat{K}.$  In view of (2.1) it then follows that for any $ f \in \mathcal{H} $ the $ \mathcal{H}$ valued function
$$ u(\rho) = \rho^s \int_K \varphi_{s,\rho}(k) \pi(k)^\ast f dk$$ solves  the following extension problem for the operator $ \pi(\Delta) $:
$$ \big( \pi(\Delta)+\partial_\rho^2 +\frac{1-s}{\rho} \partial_\rho \big) u(\rho) = 0, \,\,  u(0) = f .$$

We now specialise to the case where $ \mathcal{H} = \mathcal{H}L^2(G, \nu_t)$  which is the image of $ L^2(K) $ under the Segal-Bargmann transform.  This space is invariant under the left action of $ K $ on $ K_\C.$ Thus we can define a representation $ \pi $ of $ K $ on $\mathcal{H}L^2(G,\nu_t)$ by setting $ \pi(k)F(g) = F(k^{-1}g), g \in G.$ We denote the operator $ \pi(\Delta) $ by $ \Delta_\C $ and call it the complexified Laplacian. From the above discussions it follows that  a solution of the extension problem
\begin{equation}\label{ext} \big( \Delta_\C+\partial_\rho^2 +\frac{1-s}{\rho} \partial_\rho \big) U(g,\rho) = 0, \,\,  U(g,0) = F(g) \end{equation}
is given by the integral representation
\begin{equation} U(g,\rho) = \rho^s \int_K \varphi_{s,\rho}(k) F(kg) dk.\end{equation}
We now define $ D( (-\Delta_\C)^{s/2}) $ as  the  subspace of $ \mathcal{H}L^2(G,\nu_t)$ consisting of those $ F $ for which  $ \rho^{1-s} \partial_\rho U(g,\rho) $ with $ U $ defined as above has a limit in $\mathcal{H}L^2(G,\nu_t)$ as $ \rho \rightarrow 0.$ For any $ F \in  D( (-\Delta_\C)^{s/2}) $ we then define $ (-\Delta_\C)^{s/2}F $ as a (suitable) constant multiple of the above limit.

\begin{prop} Let $ 0 < s < 2$. For any $ F \in \mathcal{H}L^2(G, \nu_t),\, F = f \ast q_t$ we have the relation $ (-\Delta_\C)^{s/2}F(g) = (-\Delta)^{s/2}f \ast q_t(g) .$
\end{prop}
\begin{proof} Given $ F = f \ast q_t $ consider the solution $ U(g,\rho) $ of the extension problem given in (\ref{ext}).  Since $ \varphi_{s,\rho}(k) = \varphi_{s,\rho}(k^{-1}) $ (a property inherited from the heat kernel) we see that 
$ U(k,\rho) = \rho^s \big( \varphi_{s,\rho} \ast f \big)\ast q_t(k)$ for all $ k \in K.$ But the same is true for any $ g \in G:$  indeed
$$U(g,\rho) = \rho^s \int_K \varphi_{s,\rho}(k) (f\ast q_t)(kg) dk.$$
Since
$$ f\ast q_t(k g) = \int_K f(y) q_t(y^{-1}kg) dy =  \int_K f(k y) q_t(y^{-1}g) dy.$$
In view of the relation  $ \int_K \varphi_{s,\rho}(k) f(ky) dk =  \varphi_{s,\rho} \ast f(y) $ we see that
$$U(g,\rho) = \rho^s \int_K \varphi_{s,\rho} \ast f(y) q_t(y^{-1}g) dy = \int_K u(y,\rho)  q_t(y^{-1}g) dy. $$
Thus  we have  formula
$$ (-\Delta_\C)^{s/2}F(g)  = c_s  \lim_{\rho \rightarrow 0} \int_K \rho^{1-s} \partial_\rho u(y,\rho)  q_t(y^{-1}g) dy = (-\Delta)^{s/2}f \ast q_t(g) $$
where the convergence happens in $ \mathcal{H}L^2(G,\nu_t) $ under the assumption on $ F.$ This proves the proposition.
\end{proof}

We have proved this relation  $ (-\Delta_\C)^{s/2} (f\ast q_t)(g) = (-\Delta)^{s/2}f \ast q_t(g) $  for all $ 0 \leq s \leq 2.$ We claim that the same holds for  $ s = 2 $ as well. Recall that we have 
$$ \Delta\pi(k)^\ast F(g) = \frac{d}{dt}\Big{|}_{t=0} \pi(e^{-tX}k^{-1})F(g) = - \pi(\Delta)\pi(k)^\ast F(g) $$
where $ \Delta $ is applied in the $ k $ variable and $ \pi(\Delta) $ in the $ g $ variable.
A simple calculation shows that 
$  \pi(k)^\ast F(g) = \pi(k^{-1})f \ast q_t(g) $ and hence with $ R_y $ standing for right translation
$$ \Delta\pi(k)^\ast F(g) = \int_K  \Delta (R_yf(k)) q_t(y^{-1}g) dy = \int_K  (\Delta f)(ky) q_t(y^{-1}g) dy$$
where we have used the bi-invariance of $ \Delta.$ This proves that $ \pi(\Delta)(\pi(k)^\ast F)(g) = (\Delta f)\ast q_t( kg).$ Evaluating at the identity we get $ \Delta_\C (f\ast q_t) = (\Delta f)\ast q_t.$

The above has some interesting consequences. Observe that $ f\ast q_t(g) $ is an eigenfunction of $ \Delta_\C $ whenever $ f $ is an eigenfunction of $ \Delta.$ In particular, for any $ f \in L^2(K) $ and $ \pi \in \widehat{K} $ the holomorphic function $ f\ast \chi_\pi(g) $ is an eigenfunction of $ \Delta_\C $ with eigenvalue $ -\lambda_\pi^2.$ The spectral decomposition of $ \Delta_\C $ is given by
$$  -\Delta_\C F(g) = \sum_{\pi \in \widehat{K}}   d_\pi \,  \lambda_\pi^2  e^{-\frac12 \lambda_\pi^2}  f\ast \chi_\pi(g)$$
whenever $ F = f \ast q_t.$  In view of this we can define the fractional powers $ (-\Delta_\C)^{s/2} $ by spectral theorem as
$$ (-\Delta_\C)^{s/2} F(g)  = \sum_{\pi \in \widehat{K}}   d_\pi \,  \lambda_\pi^s  e^{-\frac12 \lambda_\pi^2}  f\ast \chi_\pi(g)  =  ((-\Delta)^{s/2} f)\ast q_t(g).$$
Therefore, for any  $ s \geq 0  $ we can now define $ \mathcal{H}^s(G,\nu_t) $ as the subspace of $ \mathcal{H}L^2(G,\nu_t) $ consisting of $ F $ for which $ (-\Delta_\C)^{s/2} \in \mathcal{H}L^2(G,\nu_t) .$ In view of the above  it follows that $ C_t : H^s(K) \rightarrow \mathcal{H}^s(G,\nu_t) $ is an isometric isomorphism.

Thus we have defined the holomorphic Sobolev spaces $\mathcal{H}^s(G,\nu_t) $ for  for any $ s\geq 0 $ as the image of $ H^s(K)$ under $ C_t.$  However,  we would like to come up with an intrinsic definition which does not use any information on the function $ f.$  Recall that when $ s = 2m $ the definition does not involve $ f $ explicitly as it only requires that $ (-\Delta_\C)^m F $ belongs to $ L^2(G,\nu_t).$ We are interested in coming up with such a definition for fractional order Sobolev spaces.

In order to motivate what we plan to do, observe that for any $ s, t \geq 0 $ we have the semigroup property 
$ (-\Delta_\C)^{s/2}(-\Delta_\C)^{t/2} = (-\Delta_\C)^{(s+t)/2}.$  Assuming that we can define negative powers of $ -\Delta_\C $ we expect the relation $ (-\Delta_\C)^{s/2} = (-\Delta_\C) (-\Delta_\C)^{-1+s/2}.$ This suggests that we need to look at $ (-\Delta_\C)^{-\gamma} $ and hence $ (-\Delta)^{-\gamma} $ for $ \gamma > 0 .$ Recall that $ (-\Delta)^{-\gamma} $ is defined by the gamma integral
$$ (-\Delta)^{-\gamma} f = \frac{1}{\Gamma(\gamma)} \int_0^\infty r^{\gamma-1} \, e^{r\Delta}f  \,dr.$$ 
Since constants are annihilated by $ \Delta $ they do not belong to the domain of $ (-\Delta)^{-\gamma}.$ Instead of excluding constants we consider the truncated Gamma integral
$$ R_t^\gamma f = \frac{1}{\Gamma(\gamma)} \int_0^t r^{\gamma-1} \, e^{r\Delta}f  \,dr.$$ 
Observe that $ R_t^\gamma f $ is defined fo any $ f \in L^2(K) $ and as $ t $ tends to $\infty$  it converges to $ (-\Delta)^{-\gamma}f $ for all $ f $ with integral zero.

We would like to replace $ (-\Delta_\C)^{s/2}F $ in the definition of $ \mathcal{H}^s(G,\nu_t) $ by $ (-\Delta_\C) R_t^{1-s/2} F$ where  we set $ R_t^\gamma F =  (R_t^\gamma f )\ast q_t.$ Observe that $ (-\Delta_\C) R_t^{1-s/2} F=    R_t^{1-s/2}  (-\Delta_\C) F.$ We will now show that a given function $ F \in \mathcal{H}L^2(G,\nu_t) $ belongs to the  weighted Bergman space $ \mathcal{H}L^2(G, w_{t,1-s/2})$ if and only if $ R_t^{1-s/2}F \in \mathcal{H}L^2(G,\nu_t).$ Once this is done, we will redefine $ \mathcal{H}_0^s(G,\nu_t) $ as the space of $ F \in \mathcal{H}L^2(G,\nu_t) $ such that $ \Delta_\C F \in \mathcal{H}L^2(G, w_{t,1-s/2}).$
Of course we need to check that both definitions coincide and this new definition can be extended for all $ s >0.$
We begin with our first claim. For $ \gamma > 0 $ we define the weight function
$$  w_{t,\gamma}(g)  = \frac{1}{\Gamma(2\gamma)} \int_0^t r^{2\gamma-1} \, \nu_{t-r}(g) \,dr = \frac{1}{\Gamma(2\gamma)} \int_0^t (t-r)^{2\gamma-1} \, \nu_r(g) \,dr $$ 
which is just the Riemann-Liouville  fractional integral of $ \nu_t.$

\begin{thm}  Let $ \gamma > 0 .$  For any  $ F \in  \mathcal{H}L^2(G,\nu_t), \,  R_t^\gamma F \in \mathcal{H}_0^s(G,\nu_t) $ if and only if $   F \in \mathcal{H}L^2(G, w_{t,\gamma}).$
\end{thm} 
\begin{proof} We prove this theorem by an application of Theorem 1.1. So we need to compute 
$$ \sigma_{t,\gamma} (\pi) =  d_\pi^{-1} \int_G \| \pi(g^{-1})\|^2 w_{t,\gamma}(g) dg = \frac{1}{\Gamma(2\gamma)} \int_0^t  (t-r)^{2\gamma-1} \sigma_r(\pi)  dr$$
where $ \sigma_r(\pi) $ has been already calculated:
$$ \sigma_r(\pi) = d_\pi^{-1} \int_G \|\pi(g^{-1})\|^2 \nu_r(g) =  e^{r\lambda_\pi^2}.$$
Thus $ \sigma_{t,\gamma}(\pi) $ is explicitly given by 
$$ \sigma_{t,\gamma}(\pi) = \frac{1}{\Gamma(2\gamma)}\int_0^t (t-r)^{2\gamma -1} e^{r \lambda_\pi^2} dr = e^{t\lambda_\pi^2} \frac{1}{\Gamma(2\gamma)}\int_0^t r^{2\gamma-1} e^{-r\lambda_\pi^2} dr.$$
Therefore, in view of Theorem 1.1 we have the equality
\begin{equation}  \int_G  | F(g)|^2 w_{t,\gamma}(g)  dg = \sum_{\pi \in \widehat{K}} d_\pi  \| \pi( f)\|^2  e^{-t\lambda_\pi^2} \sigma_{t,\gamma}(\pi) .\end{equation}
In view of the easily verifiable estimate
$$  C \lambda_\pi^{-4\gamma} \leq  \frac{1}{\Gamma(2\gamma)} \int_0^t  r^{2\gamma-1} e^{-r \lambda_\pi^2} dr \leq \lambda_\pi^{-4\gamma} $$ 
we conclude that the left hand side of (2.4) is finite if and only if 
$$ \sum_{\pi \in \widehat{K}} d_\pi  \| \pi(f)\|^2  \lambda_\pi^{-4\gamma} < \infty. $$
And this happens precisely when $ R_t^\gamma F $ belongs to $ \mathcal{H}L^2(G,\nu_t).$ Indeed, $ R_t^\gamma F(g) = (R_t^\gamma f)\ast q_t(g) $ and 
hence we need to check if $ R_t^\gamma f \in L^2(K).$ But 
$$ \pi(R_t^\gamma f) = \pi(f) \frac{1}{\Gamma(\gamma)} \int_0^t r^{\gamma-1} \, e^{-r\lambda_\pi^2}   \,dr.$$  As the integral behaves like $ \lambda_\pi^{-2\gamma} $ our claim is proved. Hence the theorem.
\end{proof}

\begin{cor}  For any $ 0 < s < 2 $ a function  $ F \in  \mathcal{H}L^2(G,\nu_t) $ belongs to $ \mathcal{H}_0^s(G,\nu_t) $ if and only if $  \Delta_\C F \in \mathcal{H}L^2(G, w_{t,1-s/2}).$
\end{cor} 

We are ready to define $ \mathcal{H}_0^s(G,\nu_t) $ for any $ s > 0.$ Fix any integer $ m $ such that $ 0 <s < 2m.$  We say that a function $ F $ from $ \mathcal{H}L^2(G,\nu_t)$ belongs to $
\mathcal{H}_0^s(G,\nu_t) $ if and only if $ \Delta_\C^m F \in \mathcal{H}L^2(G, w_{t,m-s/2}).$ We remark that as $ s \rightarrow 2m $ the weight function $ w_{t,m-s/2} $ converges to $ \nu_t $ and hence the new definitions coincides with the original definition given in \cite{HL}. From the above proof, we can also restate the definition as follows: $ F \in \mathcal{H}_0^s(G,\nu_t) $ if and only if $ (-\Delta_\C)^m R_t^{m-s/2}F \in \mathcal{H}L^2(G,\nu_t).$ Since $ \Delta_\C^m F $ makes sense for any $ F \in \mathcal{H}L^2(G,\nu_t) $ it is the integrability  condition, namely that  
$ \Delta_\C^m F $ is square integrable with respect to the measure $ w_{t,m-s/2}(g) dg $ that determines whether $ F \in \mathcal{H}_0^s(G,\nu_t) $ or not.

\begin{cor} For any $ s \geq 0 $ the Segal-Bargmann  transform $ C_t: H^s(K) \rightarrow \mathcal{H}_0^s(G,\nu_t) $ is an isomorphism.
\end{cor} 

\begin{rem}  Consider the operators $ (-\Delta_\C)^m R_t^{m-s/2} $ acting on  $ \mathcal{H}L^2(G,\nu_t).$  By definition, if $  F = f \ast q_t $ we have  $ (-\Delta_\C)^m R_t^{m-s/2}F = ((-\Delta)^{m}R_t^{m-s/2}f) \ast q_t.$  Thus
$$ \Delta_\C^m R_t^{m-s/2}(f \ast q_t)(g) = \frac{1}{\Gamma(m-s/2)} \int_0^t r^{m-s/2-1} (\Delta^m f \ast q_{r+t})(g) dr. $$
Using the fact that $ \frac12 \Delta(f\ast q_t) =\partial_t (f\ast q_t) $ we can rewrite the above as
$$ \Delta_\C^m R_t^{m-s/2}F(g) = \frac{2^m}{\Gamma(m-s/2)} \int_0^t  r^{m-s/2-1}  \partial_r^m(f \ast q_{t+r})(g)  dr.$$
This suggests that we consider the map $ C_t^sf(g) = f \ast q_t^s(g)$ where 
$$ q_t^s(k) =\frac{2^m}{\Gamma(m-s/2)} \int_0^t  r^{m-s/2-1} \partial_r^m q_{r}(k)  dr  $$
and study its mapping properties. In view of Theorem 1.1, a  simple calculation shows that 
$$ \int_{G} |C_t^sf(g) |^2 \nu_t(g) dg = \sum_{\pi \in \widehat{K}} d_\pi \|\pi(f)\|^2  \lambda_\pi^{4m} (\mu_{t, m-s/2}(\pi))^2  $$
where $ \Gamma(\gamma) \mu_{t,\gamma}(\pi) =  \int_0^t r^{\gamma-1} e^{-\frac12 r \lambda_\pi^2} dr$ which behaves like $ \lambda_\pi^{-2\gamma}.$ Consequently, we see that 
$$ C_1  \|f \|_{(s)}^2  \leq \int_{G} |C_t^sf(g) |^2 \nu_t(g) dg  \leq C_2 \|f \|_{(s)}^2.$$
Hence the generalised Segal-Bargmann transform $ C_t^s: H^s(K) \rightarrow \mathcal{H}L^2(G,\nu_t) $ is an isomorphism.
\end{rem} 

\section{ Holomorphic Sobolev spaces for the Hermite operator}
As we mentioned in the introduction, the proofs of Theorems 1.6 and 1.7 depend on Gutzmer's formula  which we recall now. Consider a family of unitary operators acting on $ L^2(\R^n)$ which are explicitly given by
$$ \pi(x,u)\varphi(\xi) = e^{i (x\cdot \xi +\frac12 x\cdot u)} \varphi(\xi+y),\,\, \varphi \in L^2(\R^n) $$ indexed by $ (x,u) \in \R^{2n}.$ These are related to the Schr\"odinger representation $ \pi_! $ of the Heisenberg group $ \mathbb{H}^n.$ Observe that $ \pi(z,w)F(\zeta) $ can be defined similarly for $ (z,w) \in \C^{2n} $ when $ F$ is holomorphic. However, $\pi(z,w)F(\xi) $ need not be in $ L^2(\R^n) $ unless further assumptions are made on $ F.$ Assuming that $ F $ is such a function, we are interested in finding a formula for $ \|\pi(z,w)F(\cdot)\|_{L^2(\R^n)}.$
In order to state Gutzmer's formula we need to introduce some more notation. Let $Sp(n,\R$) stand for the symplectic group consisting of $2n \times 2n$ real matrices
that preserve the symplectic form $[(x, u), (y, v)] = (u \cdot y- v \cdot x) $on $\R^{2n}$
and have determinant one. Let $O(2n,\R) $ be the orthogonal group and we define
$K = Sp(n,\R) \cap O(2n,\R).$ Then there is a one to one correspondence between $K$ and
the unitary group $U(n) .$  The action of  $ \sigma  = a + ib $ on $ \R^{2n} $ is given by $ \sigma(x,u) = ( ax-bu, bx+au) $ and this has a natural extension to $ \C^{2n}.$
We have the following  result which is known as Gutzmer's formula for the Hermite expansions, see \cite{T2}.

\begin{thm} Let $ F $ be a holomorphic function on $ \C^n $ whose restriction to $ \R^n $ is denoted by $ f. $  Then for any $ z=x+iy, w=u+iv $ in $ \C^n $ we have
$$ \int_{\R^n} \Big( \int_K |\pi(\sigma(z,w))F(\xi)|^2 d\sigma \Big) d\xi = e^{(u\cdot y-v \cdot x)} \sum_{k=0}^\infty \frac{k! (n-1)!}{(k+n-1)!} \varphi_k(2iy,2iv) \|P_kf\|_2^2 .$$

\end{thm}

In the above formula, $ P_k $ are the orthogonal projections appearing in the spectral decomposition of the Hermite operator: thus
$$   f = \sum_{k=0}^\infty  P_kf, \,\,\,\,\,\,    H = \sum_{k=0}^\infty  (2k+n) P_kf  \,\,.$$
If $ L_k^{n-1}(t) $ stand for the Laguerre polynomials of type $ (n-1) $ then the functions $ \varphi_k(z,w) $ are defined by 
$$ \varphi_k(z,w) = L_k^{n-1}(\frac12(z^2+w^2))e^{-\frac14 (z^2+w^2)}.$$
The normalised Hermite functions, indexed by $ \alpha \in \mathbb{N}^n $ are of the form $ \Phi_\alpha(x) = c_\alpha H_\alpha(x) e^{-\frac12 |x|^2} $ where $ H_\alpha $ are the Hermite polynomials. These are eigenfunctions of the Hermite operator $ H $ with eigenvalues $ (2k+n) $ and they form an orthonormal basis for $ L^2(\R^n).$ The Hermite semigroup $ T_t = e^{-tH} $ is then defined by
$$  T_t f(x)  = \sum_{k=0}^\infty  e^{-(2k+n)t} P_kf(x).$$
As $ P_kf $ is a finite linear combination of $ \Phi_\alpha, |\alpha| = k $ it is clear that it has a holomorphic extension to $ \C^n.$ It can be shown that the same is true of $ T_tf$ for all $ f \in L^2(\R^n) $ and the map taking $ f $ into $ T_tf(z) $ is known as the heat kernel transform associated to the Hermite operator. We need one more ingredient which is the analogue of (1.1) in the Hermite  setting. If we  let $$ p_t(y,v) = c_n (\sinh t)^{-n} e^{-\frac14 (\coth t) (y^2+v^2)} $$
then we have the interesting formula
\begin{equation}    \frac{k! (n-1)!}{(k+n-1)!}   \int_{\R^n} \int_{\R^n}  \varphi_k(2iy,2iv)  p_{2t}(2y,2v) dy\,dv = e^{2(2k+n)t}.\end{equation}
We refer to \cite{T1} for a proof of this and to \cite{T2} to see how it is used along with Gutzmer's formula to deduce Theorem 1.5.\\

{\it{Proof of Theorem 1.6:}} We start with the following observation: it is easily seen that the weight function $ U_t(z) $ is equal to the integral
\begin{equation} U_t(\xi+iv)  = \int_{\R^n} p_{2t}(2y,2v) e^{-2 y \cdot \xi} dy= c_n (\sinh 4t)^{-n/2} e^{-(\coth 2t)|v|^2+ (\tanh 2t)|\xi|^2} \end{equation}  Now Gutzmer's formula applied to $ T_t^sf $ gives us
$$ \int_{\R^{2n}} \Big(\int_{\R^n} \int_K |\pi(\sigma(iy,iv))T_t^sf(\xi)|^2 d\sigma d\xi \Big) p_{2t}(2y,2v)  dy dv$$  $$  =  \sum_{k=0}^\infty  e^{2t(2k+n)} (2k+n)^{2m}(a_{t,m-s/2}(k))^2 \|P_kf\|_2^2 $$
where
$$ a_{t,\gamma}(k) =  \frac{1}{\Gamma(\gamma)} \int_0^t r^{\gamma-1} e^{-(r+t)(2k+n)} dr $$ $$= e^{-t(2k+n)} \frac{1}{\Gamma(\gamma)} \int_0^t r^{\gamma-1} e^{-r(2k+n)} dr.$$
Observe that $ e^{t(2k+n)} a_{t,m-s/2}(k) $ behaves like $ (2k+m)^{-m+s/2} $ and hence $$\sum_{k=0}^\infty  e^{2t(2k+n)} (2k+n)^{2m}(a_{t,m-s/2}(k))^2 \|P_kf\|_2^2 $$ is equivalent to the norm in $ W_H^{s,2}(\R^n).$ On the other hand, as $ p_{2t}(2y,2v) $ is invariant under the action of $ K $ we have 
$$ \int_{\R^{2n}} \Big(\int_{\R^n} \int_K |\pi(\sigma(iy,iv))T_t^sf(\xi)|^2 d\sigma d\xi \Big) p_{2t}(2y,2v)  dy dv $$ 
$$ =\int_{\R^{2n}} \Big(\int_{\R^n}  |\pi((iy,iv))T_t^sf(\xi)|^2  d\xi \Big) p_{2t}(2y,2v)  dy dv .$$
Recalling the definition of $ \pi(iy,iv) $ and making use of (3.2) we see that the above integral simplifies to the square of the norm of $ T_t^sf(z) $ in $ L^2(\C^n,U_t).$ This completes the proof of Theorem 1.6.

{\it{Proof of Theorem 1.7:}}  Recall that we have defined the complexified Hermite operator as $ H_\C = \sum_{j=1}^n (-\frac{\partial^2}{\partial z_j^2} +z_j^2).$ When $ F $ is holomorphic, $ \frac{\partial}{\partial z_j} F(z) = \frac{\partial}{\partial x_j}F(z) $ and hence the restriction of $ H_\C F $ to $ \R^n $ is just $ HF.$ Consequently, if $ F $ is an eigenfunction of $ H $ having a holomorphic  extension to $ \C^n $ we see that $ F(z) $ is an eigenfunction of $ H_\C$. In particular, applied to $ P_kf(z) $ we conclude that $ H_\C P_kf(z) = P_k(Hf)(z) $ and this leads to the relation $  H_\C T_tf(z) = T_t(Hf)(z).$ 
$$  \Big(\int_{\R^n} \int_K |\pi(\sigma(iy,iv)) H_\C^m(T_tf)(\xi)|^2 d\sigma d\xi \Big) $$  $$  =  \sum_{k=0}^\infty  e^{-2t(2k+n)} (2k+n)^{2m} \frac{k! (n-1)!}{(k+n-1)!} \varphi_k(2iy,2iv)  \|P_kf\|_2^2 .$$
Integrating the above against the $K$-invariant weight function
 $$ \frac{1}{\Gamma(2m-s)} \int_0^t (t-r)^{2m-s-1} p_{2r}(2y,2v) dr $$
 we obtain the following identity:
 $$ \frac{1}{\Gamma(2m-s)} \int_0^t (t-r)^{2m-s-1}  \Big(\int_{\R^{2n}} \int_{\R^n}  |\pi((iy,iv)) H_\C^m(T_tf)(\xi)|^2  p_{2r}(2y,2v) d\xi  dy dv \Big) $$  $$  =  \sum_{k=0}^\infty  e^{-2t(2k+n)} (2k+n)^{2m}  \Big(  \frac{1}{\Gamma(2m-s)} \int_0^t (t-r)^{2m-s-1} e^{2r(2k+n)} dr \Big)\|P_kf\|_2^2 .$$
 Now the right hand side is equivalent to the square of the $ W_H^{s,2}(\R^n) $ norm of $ f $ whereas the left hand simplifies to 
 $$ \int_{\R^n} \int_{\R^n} |T_tf(\xi+iv) |^2 U_{t,m-s/2}(\xi+iv) d\xi dv $$
 where $  U_{t,m-s/2}(\xi+iv) $ is given by the integral 
 $$ \frac{1}{\Gamma(2m-s)} \int_0^t (t-r)^{2m-s-1} \Big(\int_{\R^n} e^{-2 y \cdot \xi} p_{2r}(2y,2v) dy\Big)  dr $$
 $$ = \frac{1}{\Gamma(2m-s)} \int_0^t (t-r)^{2m-s-1} U_r(\xi+iv)dr.$$
This completes the proof of Theorem 1.7.



\begin{thebibliography}{99}
\bibitem{B} D. W. Byun, {Inversions of Hermite semigroup}, \emph{Proc. Amer. Math. Soc.}, \textbf{118} (1993), 437-445.

\bibitem{CS} L. Caffarelli and L. Silvestre, {An extension problem related to the fractional Laplacian},
\textit{Comm. Partial Differential Equations}
\textbf{32} (2007), 1245--1260.


\bibitem{G} N. Garofalo, {Fractional thoughts}, arxiv.org/abs/1712.03347 (2017)

\bibitem{H} B. Hall, {Segal-Bargmann  coherent state transform for compact Lie groups},
\emph{J. Funct. Anal.}  \textbf{122} (1994), 103-154.

\bibitem{HL} B. Hall and W. Lewkeeratiyutkul, {Holomorphic  Sobolev  spaces  and  the  generalised  Segal-Bargmann  transform}
\emph{J. Funct. Anal.}, 
\textbf{217}  (2004), 192-220.

\bibitem{KTX} B. Kr\"otz, S. Thangavelu and Y. Xu, {The heat kernel transform  on the Heisenberg group}, \emph{J. Funct. Anal.}, 
\textbf{225} (2005), 301-336.

\bibitem{KOS} B. Kr\"otz, G. Olafsson and R. Stanton, { The image of the heat kernel transform on Riemannian symmetric spaces of non-compact type}, \emph{Int. Math. Res. Not.}, \textbf{22}(2005),1307-1329.


\bibitem{RT} R. Ramakrishnan and S. Thangavelu, {Holomorphic Sobolev spaces, Hermite and special Hermite semigroups and a Paley-Wiener theorem for the windowed Fourier transform}, \emph{J. Math. Anal. and Appl.},  \textbf{354} (2009), 564-574.

\bibitem{MS} M. Stenzel, {The Segal-Bargmann transform on a symmetric space of compact type}, \emph{J. Funct. Anal.} \textbf{165} (1999), 44-58.


\bibitem{ST} P. R. Stinga and J. L. Torrea,
{Extension problem and Harnack's inequality for some fractional operators},
\emph{Comm. Partial Differential Equations}
\textbf{35} (2010), 2092--2122.


\bibitem{T1} S. Thangavelu, {Gutzmer's formula and Poisson integrals on the Heisenberg group}, 
\emph{Pacific J. Math.} \textbf{231} (2007), no. 1, 217--237.

\bibitem{T2} S. Thangavelu,  {An analogue of Gutzmer's formula  for Hermite expansions}, \emph{Stud. Math.}\emph{}, \textbf{185} (2008), 279-290. 

\bibitem{T3} S. Thangavelu, {Holomorphic Sobolev spaces associated to compact symmetric spaces}, \emph{J. Funct. Anal.}, 
\textbf{251} (2007), 438-462.

\bibitem{T4} S. Thangavelu, {On the unreasonable effectiveness of Gutzmer's formula, Harmonic analysis and partial differential equations}, \emph{Contemp. Math.}, \textbf{505} (2010), 199-217, Amer. Math. Soc., Providence, RI.






\end{thebibliography}
\end{document}